\documentclass{elsarticle}

\usepackage{setspace,amsfonts,amsthm,amssymb,amsmath,comment,fancyhdr,commath}
\usepackage[mathscr]{euscript}
\usepackage[utf8]{inputenc}
\usepackage[T1]{fontenc}

\makeatletter
\newcommand{\subjclass}[2][2010]{%
  \let\@oldtitle\@title%
  \gdef\@title{\@oldtitle\footnotetext{#1 \emph\small{Mathematics subject classification.} #2}}%
}

\newtheorem{Th}{Theorem}
\newtheorem*{Th*}{Theorem}

\newtheorem{Wn}{Corollary}
\newtheorem{Le}{Lemma}

\theoremstyle{definition}

\newtheorem{Uw}{Remark}

\newcommand\mR{{\mathbb R}}

\newcommand\mN{{\mathbb N}}
\newcommand\mQ{{\mathbb Q}}

\newcommand\restr{\mathord{\restriction}}

\newcommand\B{{\mathscr B}}

\newcommand\PC{{\mathscr{P}}}

\newcommand\cn{\colon}

\newcommand\ol{\overline}

\newcommand\fs{F_\sigma}
\newcommand\gd{G_\delta}
\newcommand\fcb{\mathfrak{fcBor}}
\newcommand\flb{\mathfrak{flB}}
\newcommand\fii{{\varphi}}

\begin{document}

\begin{frontmatter}
\title{Linear extensions of Baire-one and Borel functions}

\author{Waldemar Sieg}
\ead{waldeks@ukw.edu.pl}

\address{Kazimierz Wielki University, Institute of Mathematics, Powsta\'nc\'ow Wielkopolskich 2, \\85-090 Bydgoszcz, Poland}

\begin{abstract}
Let $X$ and $Y$ be the Hausdorff topological spaces and let $A$ be both an $\fs$- and $\gd$- subset of $X$. Let also $f\cn A\to Y$ be a function
for which the inverse image of every open subset $U\subset Y$ is $\fs$ in $X$. We show that $f$ can be linearly extended to a function
with the same property defined on $X$. A similar result is proved for Baire-one function defined on an analogous subset
of $\mR$. We give also an answer when the extension map is (with a supremum norm) an isometry.
\end{abstract}



\end{frontmatter}

\section{Introduction}
Let $X$, $Y$, $A$ and $f$ have the same meaning as in the abstract.
The mapping $r\cn X\to A$ is an {\em algebraic retraction} if $r(x)=x$ for every $x\in A$.
Recall that $\B_1(X,Y)$ is the space of pointwise limits of sequences
of continuous functions $X\to Y$. We say that a function $f\cn X\to Y$ is
\begin{itemize}
 \item {\em piecewise continuous}, if there is an increasing sequence $\left(X_n\right)\subset X$ of nonempty closed sets such
that $X=\bigcup^\infty_{n=0}X_n$ and the restriction $f_{\restr X_n}$ is continuous for each $n\in\mN$;
\item {\em of the first Borel class}, if the $f$-inverse image of every open subset of $Y$ is $\fs$ in $X$ (see \cite[p. 1]{JayRog});
\item {\em of the first level Borel class}, if the $f$-inverse image of every $\fs$ subset of $Y$ is $\fs$ in $X$ (see \cite[p. 1]{JayRog}).
\end{itemize}
Note that in the last definition, the words ''$\fs$ subset of $Y$'' may be equivalently replaced by ''closed subset of $Y$''.
The symbols $\PC(X,Y)$, $\fcb(X,Y)$ and $\flb(X,Y)$ stand for the families of piecewise continuous, first Borel class and first level Borel class of functions
$X\to Y$, respectively. The inclusion $\flb(X,Y)\subset\fcb(X,Y)$
holds whenever any open set in $Y$ is $\fs$ (e.g. if $Y$ is perfectly normal). A function witnessing the fact that
$\flb(\mR,\mR)\neq\fcb(\mR,\mR)$ is the classical Riemann function (i.e. $f(x)=0$ for $x$ irrational and $f(x)=\frac{1}{q}$ if $x=\frac{p}{q}$ where
$p$, $q$ are mutually prime), which is of the first Borel class but not of the first level Borel class.
If $f\in\fcb^b(A,Y)$ and $f\in\B_1^b(A,Y)$ are bounded elements of $\fcb(A,Y)$ and $\B_1(A,Y)$, respectively, the \emph{supremum norm}
$\|\cdot\|_A$ of $f$ on $A$ is defined as $\|f\|_A=\sup_{x\in A}|f(x)|$.

This paper deals with the general problem in the theory of real functions, which is inspired by the Tietze extension theorem:

\begin{center}\emph{$(P)$ Let $A$ be a nonempty subset of a topological space $X$ and let $f_0\in\mR^A$ be a function with a certain
property $(W)$. Can $f_0$ be extended to a function $f\in\mR^X$ with the same property $(W)$?}\end{center}

Classical Tietze's theorem says that every continuous function defined on a closed subset of a normal topological space $X$ can be extended
to a function that is continuous on whole $X$. Moreover,  in 1982 Jayne and Rogers \cite{JayRog} proved the following
\begin{Th}\label{Jay_Rog}
Let $X$ and $Y$ be metric, $Y$ complete. Let $A\subset X$ and $f\in\PC(A,Y)$. There exists a set $A_1\subset X$ which is 
an intersection of $\gd$- and $\fs$- subsets of $X$ and a function $\ol{f}\in\PC(A_1,Y)$ extending $f$.
\end{Th}
Finally, in 2008 Karlova \cite[Theorem 3.2]{Karlova} showed that if $X$ is perfectly paracompact and $A$ is its complete metrizable subspace 
then every continuous function $f\cn A\to Y$ has an extension $\ol{f}\in\fcb(X,Y)$.
In this paper we prove an analogous results for Baire-one mappings and the first Borel class of functions. 
In (key) Lemma $\ref{funkcja fi}$ we construct an algebraic retraction $\fii$. Its properties allow us to formulate and prove these results.

\section{Useful lemmas}
The first lemma is a well-known fact (see \cite[Theorem 2, Section VI of $\S$31]{Kur}).

\begin{Le}\label{suma_fcb}
Let $X$ be the Hausdorff topological space. If $f,g\in\fcb(X,\mR)$ then $f+g\in\fcb(X,\mR)$.
\end{Le}
The next lemma follows from the fact that $|f|$ is a composition of $f$ and an absolute value, which is continuous.

\begin{Le}\label{modul_fcb}
Let $X$ be the Hausdorff topological space. If $f\in\fcb(X,\mR)$ then $|f|\in\fcb(X,\mR)$.
\end{Le}
From Lemmas \ref{suma_fcb} and \ref{modul_fcb} we get the following

\begin{Wn}\label{krata_liniowa}
 If $X$ is the Hausdorff topological space then $\fcb(X,\mR)$ is a linear lattice.
\end{Wn}
The next lemma is crucial for our future considerations.

\begin{Le}\label{funkcja fi}
Let $X$ be the Hausdorff topological space and let $A$ be both an $\fs$- and $\gd$- subset of $X$. Let also $g\cn X\setminus A\to A$
be a continuous map. The algebraic retraction $\varphi\cn X\to A\subset X$ given by the formula
\begin{equation}\label{fi}
\varphi(x)=\left\{\begin{array}{lll}
x & \textrm{:} & x\in A \\
g(x) & \textrm{:} & x\in X\setminus A,
\end{array} \right.
\end{equation}
belongs to the class $\PC(X,X)\cap\flb(X,X)$. Furthermore, if $X$ is perfectly normal then $\fii\in\fcb(X,X)$. Finally, if $X=\mR$ then
$\fii\in\B_1(X,X)$.
\end{Le}

\begin{proof}
We show first that $\fii\in\PC(X,X)$. Let $i_A\cn A\to A$ be an identity on $A$.
Since $A$ is both an $\fs$- and $\gd$- subset of $X$, there are increasing sequences $(F_n)$ and $(G_n)$ of closed sets such that
$\bigcup_{n}F_n=A$ and $\bigcup_{n}G_n=X\setminus A$. Obviously $F_j\cap G_k=\emptyset$ for every $j,k\in\mN$. Since $g$ is continuous,
the restrictions $i_{A_{\restr F_n}}$ and $g_{\restr G_n}$ are continuous for every $n\in\mN$. Set $H_n=F_n\cup G_n$ for every $n\in\mN$. Obviously
$\bigcup_n{H_n}=X$. Furthermore, the continuity of restrictions
$i_{A_{\restr F_n}}$ and $g_{\restr G_n}$ implies continuity of $\fii$ on every $H_n$. That completes this part of the proof.

To show that $\fii\in\flb(X,X)$ we assume $F$ to be a closed subset of $X$. We have $\fii^{-1}(F)=(A\cap F)\cup g^{-1}(F)$.
Since $A$ is $\fs$ and $F$ is closed, $A\cap F$ is $\fs$. Furthermore, $g^{-1}(F)$ is a relatively closed subset of $X\setminus A$
(which is $\fs$), hence it is again $\fs$. Thus $\fii^{-1}(F)$ is $\fs$.

If $X$ is perfectly normal then any open subset of $X$ is $\fs$. Hence, from the above proven facts, it follows that in this case
 the $\fii$-inverse image of any open set is $\fs$. Thus $\fii\in\fcb(X,X)$.

To see that $\fii\in\B_1(\mR,\mR)$ it is enough to use the well-known fact that $\fcb(\mR,\mR)=\B_1(\mR,\mR)$ (cf. the Lebesgue-Hausdorff theorem,
\cite[$\S$31, Section IX]{Kur}).
\end{proof}

\section{The main results}
Our first main result is about extending functions from the first Borel class.

\begin{Th}\label{rozsz_fcb}
Let $X$ be a perfectly normal topological space and let $A$ be its both an $\fs$- and $\gd$- subset. Moreover, let $f\in\fcb(A,\mR)$
and $g\cn X\setminus A\to A$ be a continuous map. There exists a non-negative, preserving unity and linear extension operator 
$\fii^\star\cn\fcb(A,\mR)\to\fcb(X,\mR)$ such that its restriction to
$\fcb^b(A,\mR)$ is an isometry.
\end{Th}

\begin{proof}
Set $\fii^\star(f)=f\circ\fii$, where $\fii$ is an algebraic retraction given by (\ref{fi}).
Positivity of $\fii^{\star}$ is due to the fact that the class $\fcb(X,\mR)$ equipped with the order by the coordinates
is a linear lattice (see Corollary \ref{krata_liniowa}).
We shall show that $\fii^{\star}(f)\in\fcb(X,\mR)$ for every $f\in\fcb(A,\mR)$. Let $U$ be an open subset of $\mR$. There is
$$
\left(\fii^{\star}(f)\right)^{-1}(U)=(f\circ \fii)^{-1}(U)=\fii^{-1}\left(f^{-1}(U)\right).
$$
Since $f\in\fcb(A)$, $f^{-1}(U)$ is $\fs$ in $A$. Hence $f^{-1}(U)=\bigcup_{n}K_n$ with every $K_n$ closed in $A$. Thus
$$
f^{-1}(U)=\bigcup_{n}K_n=\bigcup_{n}(A\cap G_n)=A\cap\bigcup_{n}G_n
$$
with every $G_n$ closed in $X$. Therefore
$$
\left(\fii^{\star}(f)\right)^{-1}(U)=\fii^{-1}\left(A\cap\bigcup_{n}G_n\right)=\fii^{-1}(A)\cap\fii^{-1}\left(\bigcup_{n}G_n\right)=
$$
$$
=X\cap\bigcup_{n}\fii^{-1}(G_n)=\bigcup_{n}\fii^{-1}(G_n).
$$
According to Lemma \ref{funkcja fi}, every set $\fii^{-1}(G_n)$ is $\fs$ in $X$. Thence a set
$\left(\fii^{\star}(f)\right)^{-1}(U)=\bigcup_{n}\fii^{-1}(G_n)$ is $\fs$ in $X$ as well.
Thus $\fii^{\star}(f)\in\fcb(X)$.
$\fii^{\star}$ is of course an extension of $f$. Its linearity and preserving of the unity are obvious.
Moreover, if $f$ is a bounded element of $\fcb(A,\mR)$, we have
$$
\|\fii^{\star}(f)\|_X=\sup_{x\in X}|f\left(\fii(x)\right)|=\sup_{a\in A}|f(a)|=\|f\|_A,
$$
which means that $\fii^{\star}$ is an isometry.
\end{proof}

Note that in \cite[Theorem 2]{SiegWojtowicz} we have shown that if $X$ is normal and $A$ is both an $\fs$- and $\gd$- subset of $X$ then $f\in\PC(A,\mR)$ has
a linear extension $\fii^{\star}(f)=f\circ\fii$. In 2005 Kalenda and Spurn\'y showed \cite[Example 10]{KalendaSpurny} that there is a bounded Baire-one function
$f\cn\mQ\cap[0,1]\to\mR$ which cannot be extended to a Baire-one function on $[0,1]$. From these considerations and by Theorem \ref{Jay_Rog} we obtain the following

\begin{Uw}
Let $A$ be a nonempty subset of a metric space $X$. Let also $f\in\PC(A,\mR)$. If $\ol{f}\in\PC(A_1,\mR)$ extends $f$
then in Theorem \ref{Jay_Rog} we cannot require $A_1$ to be simultanoeusly $\fs$ and $\gd$.
\end{Uw}

\begin{proof}
Let $A=[0,1]\cap\mQ$ and $X=\mR$. $A$ is of course an $\fs$-subset of $X$. Moreover, $\PC(A,\mR)=\B_1(A,\mR)=\mR^A$. Let $f\cn A\to\mR$. Suppose that there is
an extension $\ol{f}\in\PC(A_1,\mR)$ such that $A_1\supset A$ is both an $\fs$- and $\gd$- subset of $\mR$ and $\ol{f}_{\restr A}=f$. Then, by
\cite[Theorem 2]{SiegWojtowicz}, the mapping $\fii^{\star}(\ol{f})=\ol{f}\circ\fii$ is a piecewise continuous (and thus Baire-one) extension of $\ol{f}$.
Those arguments would imply that every function $f\in\B_1(A,\mR)$ has an extension $\ol{f}\in\B_1(\mR,\mR)$, a contradiction.
\end{proof}
The second main theorem deals with extensions of Baire-one functions.

\begin{Th}\label{rozsz_B1}
Let $A$ be both an $\fs$- and $\gd$- subset of $\mR$ and let $f\in\B_1(A,\mR)$.
There exists a non-negative, preserving unity and linear extension operator $\fii^{\star}\cn\B_1(A,\mR)\to\B_1(\mR,\mR)$
such that its restriction to $\B_1^b(A,\mR)$ is an isometry.
\end{Th}

\begin{proof}
Set $\fii^{\star}(f)=f\circ\fii$, where $\fii$ is an algebraic retraction (\ref{fi}).
Since $f\in\B_1(A,\mR)$, there is a sequence $\left(f_n\right)$ of continuous mappings on $A$ such that $f(a)=\lim_{n\to\infty}f_n(a)$, for every
$a\in A$. Since (by Lemma \ref{funkcja fi}) $\fii\in\B_1(A,\mR)$, we have $\fii=\lim_{n\to\infty}\phi_n$ with every $\phi_n$ continuous on $A$.
Obviously every composition $f_n\circ\phi_n$ is continuous on $A$. Moreover, by Lemma \ref{funkcja fi}, for every $x\in X$ there is $n_x\in\mN$ such that
$\phi_n(x)=\fii(x)$ for every $n>n_x$. Therefore

$$
\left(f_n\circ\phi_n\right)(x)=f_n\left(\phi_n(x)\right)=f_n\left(\fii(x)\right)\to f\left(\fii(x)\right)
$$
and $f_n\circ\phi_n\xrightarrow{n\to\infty}f\circ\fii=\fii^{\star}\in\B_1(\mR,\mR)$.
Linearity and preserving of the unity are obvious for $\fii^{\star}$.
Moreover, if $f$ is a bounded element of $\B_1(A,\mR)$, we have
$$
\|\fii^{\star}(f)\|_X=\sup_{x\in X}|f\left(\fii(x)\right)|=\sup_{a\in A}|f(a)|=\|f\|_A,
$$
which means that $\fii^{\star}$ is an isometry.
\end{proof}
To conclude the paper, let us give another method of extending functions of the first Borel class.
It has a weaker assumption about $X$ and is a generalization of Theorem \ref{rozsz_fcb}.

\begin{Th}\label{rozsz_fcb_lepsze}
Let $X$ be the Hausdorff topological space and let $A$ be both an $\fs$- and $\gd$- subset of $X$. Let also $f\in\fcb(A,\mR)$. There exists 
a linear extension operator $T\cn\fcb(A,\mR)\to\fcb(X,\mR)$ such that its restriction to $\fcb^b(A,\mR)$ is an isometry.
\end{Th}

\begin{proof}
Fix any $x_0\in A$. For any $f\cn A\to\mR$ define $\ol{f}\cn X\to\mR$ by
$$
\ol{f}(x)=\left\{\begin{array}{ll}
f(x), & x\in A, \\
f(x_0), & x\in X\setminus A.
\end{array} \right.
$$
It is obvious that $T(f)=\ol{f}$ is linear and $\|f\|_A=\|\ol{f}\|_X$ whenever $f$ is bounded. Moreover, if $f\in\fcb(A,\mR)$ and $U\subset\mR$
is open then
$$
\ol{f}^{-1}(U)=\left\{\begin{array}{ll}
f^{-1}(U), & f(x_0)\notin U, \\
f^{-1}(U)\cup(X\setminus A), & f(x_0)\in U.
\end{array} \right.
$$
Since $f^{-1}(U)$ is a relatively $\fs$-subset of an $\fs$-set $A$, it is $\fs$ too. Moreover, since $X\setminus A$ is $\fs$, 
$\ol{f}^{-1}(U)$ is in both cases $\fs$. Hence $\ol{f}\in\fcb(X,\mR)$.
\end{proof}

\end{document}